\documentclass[11pt,reqno]{amsart}
\usepackage{microtype}
\usepackage{xcolor}
\usepackage[shortlabels]{enumitem}
\usepackage{amssymb}
\setlist{nolistsep}

\newtheorem{thm}{Theorem}[section]

\newtheorem{cor}[thm]{Corollary}
\newtheorem{prop}[thm]{Proposition}
\theoremstyle{definition}

\numberwithin{equation}{section}

\title[(Strong) Numerical radius orthograph]{Orthograph related to mutual (strong) numerical radius orthogonality on operator algebra}
\author{Shaoxing Sun, Jinchuan Hou, Xiaofei Qi}
\address
[Sun]{
  Fundamentals Department, Air Force Engineering University, Xi'an, 710051, P. R. China}

\address[Hou]{
  College of Mathematics, Taiyuan University of Technology, Taiyuan 030024,
P.R. China}

\address[Qi]{
	School of Mathematics and Statistics, Shanxi University, Taiyuan 030006,
	P. R. China; Key Laboratory of Complex Systems and Data Science of Ministry of Education,
	Shanxi University, Taiyuan  030006,  Shanxi, China}

\email{ssx15603406512@163.com;
  xiaofeiqisxu@aliyun.com;
  jinchuanhou@aliyun.com}

\thanks{This work is partially supported by National Natural Science Foundation of
  China (12571138, 12071336) and the Natural Science Basic Research Program of Shaanxi Province  (S2025-JC-QN-1854).}

\begin{document}

\begin{abstract}
  Let $\mathcal{B} (H)$ and $\mathcal{B} _s(H)$ be the algebra of all bounded linear operators on a complex Hilbert space $H$ and the Jordan algebra of all self-adjoint operators in $\mathcal{B} (H)$, respectively. In this paper,  the orthographs  by (strong) numerical radius orthogonality   relation on $\mathcal{B} (H)$ or $\mathcal{B} _s(H)$ is  introduced, and then the graph-theoretic properties of these orthographs, such as isolated vertices, connected components and the  diameters, are discussed.
\end{abstract}

\maketitle

\section{Introduction}


Assume that $\mathcal{X}$ is any normed  space over the real or complex  field $\mathbb{F} (=\mathbb R$ or $\mathbb C$) with some orthogonality $\perp$. We say that  $A,B\in\mathcal{X} $ are mutual orthogonal to each other, denoted by $A\perp\!\!\!\!\perp B$, if $A\perp B$ and $B\perp A$. For an orthograph $G=\varGamma (\mathcal{X} ,\perp\!\!\!\!\perp)$,  the set of its  vertices is  $V(G)=\{[A]=\mathbb{F} A : A\in \mathcal{X} {\rm \ and \ }\|A \|=1 \}$; and  the set of its edges is $E(G)=\{AB:A,B\in\mathcal{X}  {\rm \ and \ }A\perp\!\!\!\!\perp B\}$.
For the convenience of expression,  write $A$ instead of $[A]$.  Two vertices $A,B\in V(G)$ are said to be adjacent if $AB\in E(G)$, and $A\in V(G)$ is called an isolated point if it is not adjacent to any other vertices in $V(G)$. A vertex sequence $A_1A_2\cdots A_n$ is called a path if the vertices $\{A_i\}$ are distinct and $A_iA_{i+1}\in E(G)$ with $i=1,2,\cdots n-1$; and in this case, the number of edges is called the length of the path. The graph $G$  is said to be connected if there exists a path from any vertex to any other vertex of $G$. A connected component of $G$ is a maximal connected subgraph. The distance between two distint vertices is the length of the shortest path between them. The diameter of  $G$ is the maximum of distances between vertices for all pairs of vertices in $G$, denoted by ${\rm Diam}\,G$. For more details, see \cite{BL}.

In \cite{B}, Birkhoff James orthogonality was introduced, where $x\in\mathcal{X}$ is said to be Birkhoff James orthogonal to $y\in\mathcal{X}$, denoted by $x\bot _By$, if $\|x+\lambda y \| \geq \|x \| $ holds for all $\lambda \in\mathbb{F}$. So $x\perp\!\!\!\!\perp_By$ means that $x\bot _By$ and $y\bot _Bx$. In \cite{AG}, the authors proved that the orthograph  $ \varGamma (\mathcal{X} ,\perp\!\!\!\!\perp_B)$  is connected when $\dim \mathcal{X}\geq 3$, and that Diam$\varGamma (\mathcal{X},\perp\!\!\!\!\perp_B)=2$ when  $\dim\mathcal{X}\geq 5$. Later, strong Birkhoff James orthogonality   was   introduced  for C$^*$-algebras in \cite{AR}, where $A\in\mathcal{A}$ (any C$^*$-algebra) is said to be strong Birkhoff James orthogonal to $B\in\mathcal{A}$, denoted by $A\bot ^s_BB$, if $\|A+ BC \| \geq \|A \| $ holds for all $C \in\mathcal{A}$. In \cite{AGKR},  the authors discussed  the connected components and the diameter of  the orthograph $\varGamma (\mathcal{A},\perp\!\!\!\!\perp^s_B)$  when $\mathcal A$ is commutative  or $\mathcal A=\mathcal{B} (H)$ (the algebra of all bounded linear operators on a complex Hilbert space $H$), and they  proved that   $A\in\mathcal{B} (H)$ is an isolated vertex of $\varGamma (\mathcal{B} (H),\perp\!\!\!\!\perp^s_B)$ if and only if $A$ is right invertible. In \cite{KS}, the authors proved that, for any C$^*$-algebra $\mathcal{A}$,  $A\in\mathcal{A}$ is an isolated vertex of $\varGamma (\mathcal{A},\perp\!\!\!\!\perp^s_B)$  if and only if $A$ is approximately right invertible; and   except $\mathbb{C} $, $\mathbb{C} \oplus \mathbb{C} $ and $\mathcal M_2(\mathbb{C} )$ (the algebra of all $2\times 2$ complex matrices),  all non isolated points make a single connected component of  $\varGamma (\mathcal{A},\perp\!\!\!\!\perp^s_B)$ with  the diameter of this component  not exceed $4$. In \cite{ S}, the orthograph $\varGamma (\mathcal{A},\perp\!\!\!\!\perp^s_B)$  with $\mathcal A$  being  any finite-dimensional C$^*$-algebra is discussed, and showed that the  diameter of $\varGamma (\mathcal{A},\perp\!\!\!\!\perp^s_B)$  can be used to identify different types of finite-dimensional C$^*$-algebras,  that is, ${\rm Diam}\,\varGamma(\mathbb{C} \oplus \mathbb{C} , \perp\!\!\!\!\perp^s_B)=1$; ${\rm Diam}\,\varGamma(\mathbb{C} \oplus M_2(\mathbb{C} ), \perp\!\!\!\!\perp^s_B)=4$ and ${\rm Diam}\,\varGamma(\mathcal{A} , \perp\!\!\!\!\perp^s_B)=3$ for all other cases.

The numerical radius, which is equivalent to the operator norm,  is an important tool for studying operators and matrices, and has been extensively studied.
Recall that the numerical range and the numerical radius of any operator $A\in\mathcal{B} (H)$ are defined respectively  $W(A)=\{\langle Ax,x\rangle :x\in H,\| x\| =1\}$ and $w(A)=\sup \{|\lambda |:\lambda \in W(A)\}$. The authors in \cite{MPS} replaced the operator norm in B-J orthogonality with the numerical radius and introduced the concept of numerical radius orthogonality.  Recall that $A\in\mathcal{B} (H)$ is said to be numerical radius orthogonal to $B\in\mathcal{B} (H)$, denoted by $A\bot _wB$, if $w(A+\lambda B)\geq w(A)$ holds for all $\lambda \in\mathbb{C} $. Note that $\bot _B$ and $\bot _w$ are not comparable. For the study of numerical radius orthogonality, see  \cite{AK,GMPS,RS} and the references therein.
Recently, the authors in  \cite{SQH} introduced  strong numerical radius orthogonality  on $\mathcal{B} (H)$, where $A\in\mathcal{B} (H)$ is  strong numerical radius orthogonal to $B\in\mathcal{B} (H)$, denoted by $A\bot^s_w B$, if $w(A+BC)\geq w(A)$ holds for all $C\in\mathcal{B} (H)$. It is easy to verify that $A\bot^s_w B$ if and only if $ A\bot_w BC$ holds for all $C\in\mathcal{B} (H)$. The numerical radius orthographs   $\varGamma (\mathcal{B} (H), \perp\!\!\!\!\perp_w)$ and $\varGamma (\mathcal{B}_s (H), \perp\!\!\!\!\perp_w)$ were discussed in \cite{AG}. It was shown
(ref. \cite[Theorem 4.1]{AG}) that both graphs are connected, that is, they have no isolated vertices, for $\dim H\geq 2$. It was also shown that  Diam $\varGamma (\mathcal{B} (H), \perp\!\!\!\!\perp_w)={\rm Diam}\ \varGamma (\mathcal{B}_s (H), \perp\!\!\!\!\perp_w)=2$ when $\dim H\geq 3$. Then an interesting question is to ask what are the diameters of these two graphs  if $\dim H=2$?


The purpose of the present  paper is to discuss the numerical radius orthograph $\varGamma (\mathcal{Y},\perp\!\!\!\!\perp_w)$ and strong numerical radius orthograph $\varGamma (\mathcal{Y},\perp\!\!\!\!\perp_w^s)$, where $\mathcal{Y}$ represents $\mathcal{B} (H)$ or $\mathcal{B} _s(H)$ (the Jordan algebra of all self-adjoint operators in $\mathcal{B} (H)$). The rest of the paper is organized as follows.  Section 2 provides some  properties  about the (strong) numerical radius orthogonality of bounded linear operators.  Section 3 is devoted to answering the question above. We show that   ${\rm Diam}\,\varGamma (\mathcal{B} (H),\perp\!\!\!\!\perp_w)=3$ and ${\rm Diam}\,\varGamma (\mathcal{B}_s (H),\perp\!\!\!\!\perp_w)=2$ for the case $\dim H=2$ (Theorem 3.1).  For a subset $\mathcal S$ of a C$^*$-algebra $\mathcal A$, denote by NI$(\mathcal S)$ the set of all element in $\mathcal S$ that is not invertible in $\mathcal A$. In Section 4,  we consider the case of strong numerical radius orthogonality and characterize the connected components of $\varGamma (\mathcal M_n(\mathbb C),\perp\!\!\!\!\perp^s_w)$ (Theorems 4.2 and   4.5) and show that Diam\,$\varGamma ({\rm NI}(\mathcal M_n(\mathbb C)),\perp\!\!\!\!\perp^s_w)\leq 4$ for $n\geq 3$ (Theorem 4.3) and Diam$\varGamma ({\rm NI}(\mathcal M_2(\mathbb C)),\perp\!\!\!\!\perp^s_w)=4$ (Theorem 4.8).
In Section 5, we discuss the orthograph  $\varGamma (\mathcal{B}_s (H),\perp\!\!\!\!\perp^s_w)$. We characterize the connected components of    $\varGamma (\mathcal{B}_s (H),\perp\!\!\!\!\perp^s_w)$ for $\dim H\geq 3$ (Theorem 5.3) and
show that Diam\,$\varGamma ({\rm NI}(\mathcal{B}_s (H)),\perp\!\!\!\!\perp^s_w)=3$ for $\dim H\geq 4$ and Diam\,$\varGamma ({\rm NI}(\mathcal{B}_s (H)),\perp\!\!\!\!\perp^s_w)=4$ for $\dim H=3$ (Theorem 5.2). The  structure of connected components of $\varGamma (\mathcal{B}_s (H),\perp\!\!\!\!\perp^s_w)$ when $\dim H=2$ is very different which is described in  Theorem \ref{prop5.6}.

We close the section by fixing some other notions. In this paper we always assume that $H$ is any complex Hilbert space with $\dim H\geq 2$. \if false Denote by $\mathcal{B} (H)$ the algebra of all bounded linear operators on $H$ and  $\mathcal{B} _s(H)$ the Jordan algebra of all self-adjoint operators in $\mathcal{B} (H)$.\fi For $A\in\mathcal{B} (H)$, $\ker A$ and ${\rm ran}\,A$ respectively stand for the kernel and the  range of $A$. For any subspace $H'$ of $H$, $P_{[H']}$ denotes the projection operator onto $H'$. Write  ${\rm sphere}\, H=\{x\in H:\| x\| = 1 \}$.    For any set  $M$, ${\rm co}\,(M)$ denotes its convex hull.  If $H$ is finite dimensional with $\dim H=n$, then $\mathcal B(H)$ can be regarded as the algebra of $n\times n$ complex matrices  $\mathcal{M} _n(\mathbb{C} )$, and  in this case, let $\mathsf{H} _n$ denote the Jordan algebra of all self-adjoint matrices in $\mathcal{M} _n(\mathbb{C} )$.  For any $\alpha \in\mathbb{C} $, $\arg \alpha $ denotes its argument.

\section {Properties of (strong) numerical radius orthogonality}

 In this section, we list some  properties of   (strong) numerical radius orthogonality that are useful for our purpose.

The following proposition gives a characterization of numerical radius orthogonality, which  plays a key role in this paper.

\begin{prop}\label{1.1} {\rm (\cite{MPS})}
	Assume that
 $T,A\in\mathcal{B} (H)$. The following statements are equivalent.

 {\rm (1)} $T\perp _wA$.

 {\rm (2)} $\alpha T\perp _w\beta A$ for all $\alpha ,\beta \in\mathbb{C} \setminus \{0\}$.

 {\rm (3)} For each $\theta \in[0,2\pi)$, there exists a sequence $\{x_n^\theta \}\subseteq {\rm sphere}\ H=\{x\in H:\| x\|=1 \}$ such that $\lim\limits _{n\rightarrow \infty }|\langle Tx_n^\theta ,x_n^\theta \rangle |=w(T)$ and $\lim\limits _{n\rightarrow \infty }{\rm Re}\{e^{-i\theta }\langle Tx_n^\theta ,x_n^\theta \rangle\overline{\langle Ax_n^\theta ,x_n^\theta \rangle}  \}\geq 0$.
\end{prop}


The following several propositions give some sufficient (or necessary) conditions for the strong numerical radius orthogonality.

\begin{prop}\label{prop2.0}{\rm (\cite[Proposition 3.2]{SQH})} Assume that $A,B\in\mathcal{B} (H)$. If there exists a sequence  $\{x_n\}\subset  {\rm sphere}\,H$ such that $\lim\limits _{n\rightarrow \infty }|\langle Ax_n,x_n\rangle |=w(A)$ and $\lim\limits _{n\rightarrow \infty }\| B^*x_n\|=0$, then $A\bot_w^s B$.
  \end{prop}


  \begin{prop}\label{prop2.1}{\rm (\cite[Proposition 3.3]{SQH})}
  	 Assume that $A,B\in\mathcal{B}(H)$ with $0\not=A$   normaloid (that is, $\| A\| =w(A)$). Then the following statements are equivalent.
  	
   {\rm 	(i)}  $A\bot ^s_w B$.
  	
   {\rm 	(ii)}  There exist a sequence $\{x_n\}\subset {\rm sphere}\,H$ such that $\lim\limits_{n \to \infty}|\langle Ax_n,x_n\rangle |=w(A) $ and $\lim  	\limits_{n \to \infty}\| B^*x_n\| =0$.

\end{prop}


 Proposition \ref{prop2.1} does not hold if the assumption that $A$ is normaloid is removed.

\begin{prop}\label{prop2.01}{\rm (\cite[Proposition 3.6]{SQH})}
  	 Assume that $T,A\in\mathcal{M} _2(\mathbb{C} ) $ with $T\not=0$.

{\rm (i)} If $T^2=0$, that is, $T=\lambda e_1\otimes e_2$  for some $\lambda \in\mathbb{C} $ and $e_1,e_2\in{\rm sphere}\,\mathbb{C} ^2$ with $\langle e_1,e_2\rangle=0 $, and if $A=u\otimes v$ is of rank one with $u,v\in\mathbb{C} ^2$ such that  $|\langle u,e_1\rangle |<|\langle u,e_2\rangle |$, then  $T\bot ^s_wA$, but for any $x\in M_T$, $A^*x\not=0$.

{\rm (ii)} If $T^2=0$, that is, $T=\lambda e_1\otimes e_2$  for some $\lambda \in\mathbb{C} $ and   $e_1,e_2\in{\rm sphere}\,\mathbb{C} ^2$ with $\langle e_1,e_2\rangle=0 $, and if $A=u\otimes v$ is of rank one with $u,v\in\mathbb{C} ^2$ such that  $|\langle u,e_1\rangle |\geq |\langle u,e_2\rangle |$, then $T\bot _w^sA$ if and only if there exists some  $x\in M_T$ such that $A^*x=0$.

{\rm (iii)} If $T^2\not=0$, then $T\bot _w^sA$ if and only if there exists some  $x\in M_T$ such that $A^*x=0$.
Here, $M_T=\{x\in{\rm sphere}\,\mathbb{C} ^2:|\langle Tx,x\rangle |=w(T)
\}$.
\end{prop}

Particularly, for square-zero matrices, Proposition \ref{prop2.01} can be reformulated as follows.

\begin{cor}\label{prop2.011}
  	 Assume that $0\neq B\in\mathcal{M} _2(\mathbb{C} ) $, $x,y\in {\rm sphere}\,\mathbb{C} ^2$ with $\langle x,y\rangle=0$. Then the following statements are equivalent.

{\rm (1)} $x\otimes y\perp^s_wB$.

{\rm (2)} $B^*x$ and $B^*y$ are linearly dependent and $\|B^*y \| \geq \|B^*x \| $.
\end{cor}

In fact,   Corollary \ref{prop2.011} also holds  for any  rank-one nilpotent matrices in $\mathcal{M} _n(\mathbb{C} )$.

\begin{prop}\label{prop4.1}
Let $n\geq 3$. Assume that $0\neq A\in\mathcal{M} _n(\mathbb{C} )$,  $x,y\in {\rm sphere}\,\mathbb{C} ^n$ with $\langle x,y\rangle =0$. Then the following statements are equivalent.

{\rm (1)} $x\otimes y\perp^s_wA$.

{\rm (2)} $A^*x$ and $A^*y$ are linearly dependent and $\|A^*y \| \geq \|A^*x \| $.
\end{prop}

\begin{proof}
Take any $A,B\in\mathcal{M} _n(\mathbb{C} )$ with $A\not=0$ and any $x,y\in {\rm sphere}\,\mathbb{C} ^n$ with $\langle x,y\rangle =0$.
  It is easy to verify that
  \begin{equation}\label{2.01}
    \{z\in{\rm sphere}\,\mathbb{C} ^n: |\langle (x\otimes y)z,z\rangle |=w(x\otimes y)\}=\{\frac{1}{\sqrt{2}}(x+e^{i\delta }y):\delta \in[0,2\pi)\}.
  \end{equation}
    Write  $z_\delta =\frac{1}{\sqrt{2}}(x+e^{i\delta }y)$, and then $\langle (x\otimes y)z_\delta ,z_\delta \rangle =\frac{1}{2}e^{i\delta }$. A direct calculation yields
  $$\langle ABz_\delta ,z_\delta \rangle =\frac{1}{2}(\langle ABx,x\rangle +e^{-i\delta }\langle ABx,y\rangle +e^{i\delta }\langle ABy,x\rangle +\langle ABy,y\rangle ),$$   which means
  $$\overline{\langle (x\otimes y)z_\delta ,z_\delta \rangle}\langle ABz_\delta ,z_\delta \rangle=\frac{1}{4}[e^{-2i\delta }\langle Bx,A^*y\rangle+e^{-i\delta }(\langle Bx,A^*x\rangle+\langle By,A^*y\rangle  )+\langle By,A^*x\rangle  ] .$$
Define a function $f_B$ on $\mathbb C$ by
$$f_B(u)=\langle Bx,A^*y\rangle u^2+(\langle Bx,A^*x\rangle+\langle By,A^*y\rangle  )u+\langle By,A^*x\rangle.$$ By Proposition \ref{1.1}, one has  that
  \begin{equation}\label{2.1}
   x\otimes y\perp_w AB {\rm \ if \ and \  only \ if\ }0\in{\rm co}\,\{f_B(e^{i\theta }):\theta \in[0,2\pi)\}.
  \end{equation}

Let $N_B$ be the number of elements in the set $\{u:f_B(u)=0{\rm \ and \ }|u|<1\}$. Note that, according to the argument principle,  the change in the argument of  $f_B(u)$ is $2\pi N_B$ after the independent variable $u$ makes a complete counterclockwise (positive direction) circuit along $\{e^{i\theta }:\theta \in[0,2\pi)\}$.  Thus,  we get
\begin{equation}\label{2.011}
   0\in{\rm co}\,\{f_B(e^{i\theta }):\theta \in[0,2\pi)\} {\rm \ if \ and \  only \ if\ }N_B\geq 1.
  \end{equation}

(2)$\Rightarrow$(1): Assume that $A^*x$ and $A^*y$ are linearly dependent.  If
$\|A^*y \| >\|A^*x \| >0$, writing
$A^*x=\lambda A^*y$ with $\lambda \in\mathbb{C} $ and $0<|\lambda |<1$,  then
  $$f_B(u)=\langle Bx,A^*y\rangle u^2+(\bar \lambda\langle Bx,A^*y\rangle+\langle By,A^*y\rangle  )u+\bar \lambda\langle By,A^*y\rangle.$$
  This implies $f_B(-\bar\lambda)=0$. Now, using $|-\bar\lambda|<1$ and Eqs.(\ref{2.1})-(\ref{2.011}), we infer that $x\otimes y\perp_w AB$  for all $B\in\mathcal{M} _n(\mathbb{C} )$, which means $x\otimes y\perp^s_w A$.

 If  $\|A^*x \| =\|A^*y \| \geq 0$, then there exists some $\delta \in[0,2\pi)$ such that $A^*z_\delta=\frac{1}{\sqrt{2}}(A^*x+e^{i\delta }A^*y) =0$, which and Proposition \ref{prop2.0} imply $x\otimes y\perp^s_w A$.

If $A^*x=0$ and $A^*y\neq 0$, then $f_B(u)=\langle Bx,A^*y\rangle u^2+\langle By,A^*y\rangle  u$, which yields $f_B(0)=0$, which and Eqs.(\ref{2.1})-(\ref{2.011}) imply $x\otimes y\perp^s_w A$.

 (1)$\Rightarrow$(2):
If  $A^*x$ and $A^*y$ are linearly independent, then  we can
take $B_0$ such that $\langle B_0x,A^*y\rangle=\langle B_0x,A^*x\rangle=\langle B_0y,A^*y\rangle=0$ and $\langle B_0y,A^*x\rangle\neq 0$.  So $f_{B_0}(u)=\langle B_0y,A^*x\rangle\not=0$. By  Eq.(\ref{2.1}), we  have $x\otimes y\not\perp_w AB_0$, which implies $x\otimes y\not\perp^s_w A$. Thus $A^*x$ and $A^*y$ must be linearly denpendent.

If $A^*y=0$ and $A^*x\neq 0$, then  $f_B(u)=\langle Bx,A^*x\rangle u+\langle By,A^*x\rangle$. By taking $B_0$ such that $|\langle B_0y,A^*x\rangle|>|\langle B_0x,A^*x\rangle|$, one gets that $x\otimes y\not\perp_w AB_0$ by Eqs.(\ref{2.1})-(\ref{2.011}). Hence $x\otimes y\not\perp^s_w A$.
If $A^*x$ and $A^*y$ are linearly dependent and $\|A^*x \| >\|A^*y \| >0$, then
take $B_0$ such that $\langle B_0x,A^*y\rangle=\langle B_0x,A^*x\rangle=0$ and $0<|\langle B_0y,A^*y\rangle|<|\langle B_0y,A^*x\rangle|$. Thus $f_{B_0}(u)=\langle B_0y,A^*y\rangle  u+\langle B_0y,A^*x\rangle$. It follows from  Eqs.(\ref{2.1})-(\ref{2.011}) that $x\otimes y\not\perp_w AB_0$. Hence $x\otimes y\not\perp^s_w A$.  So we must have $\|A^*x\|\leq\|A^*y\|$, as desired.

The proof is finished.
\end{proof}

The following two propositions characterize invertible operators by the strong numerical radius orthogonality.

\begin{prop}\label{prop2.3} {\rm (\cite[Proposition 4.2]{SQH})}
Assume that $B\in\mathcal{B}(H)$. Then the following statements are equivalent.

   {\rm   (i)}   For any $A\in\mathcal{B} (H)$, we have $A\perp ^s_wB\Rightarrow A=0$.

  {\rm   (ii)}   $0\notin\sigma _{r}(B)$, that is, $B$ is right invertible.
   \end{prop}

 \begin{prop}\label{prop2.4} {\rm (\cite[Corollary 4.3]{SQH})}
Assume that $B\in\mathcal{B}_s (H)$.  Then the following statements are equivalent.

  {\rm   (i)}  For any $A\in\mathcal{B} _s(H)$, we have $A\perp ^s_wB\Rightarrow A=0$.

 {\rm   (ii)}    $0\notin\sigma (B)$, that is,  $B$ is invertible.

 \end{prop}

By \cite[Proposition 2.4]{SQH}, numerical radius orthogonality and B-J orthogonality are equivalent  for self-adjoint operators. However, the following proposition shows that strong numerical radius orthogonality and strong B-J orthogonality are not equivalent, and the strong numerical radius orthogonality is stronger.

\begin{prop}\label{prop3.1}
If $A, B\in\mathcal{B} (H)$ with $A$ being normaloid, then
$$A\perp ^s_wB\Rightarrow A\perp ^s_BB $$
In addition, there exist some    $A,B\in\mathcal{B} _s(H)$  such that $A\bot^s_B B$ and $A \not\perp^s_w B$.

\end{prop}
\begin{proof}

  Assume that $A\in\mathcal{B} (H)$ is normaloid. If $A\perp ^s_wB$ for some $B\in\mathcal B(H)$,  by Proposition \ref{prop2.1}, there exists a sequence $\{x_n\}\subset {\rm sphere}\,H$ such that
  $$\lim_{n\to \infty} |\langle Ax_n,x_n\rangle |=w(A)=\|A\|  {\rm \ and \ } \lim_{n\to \infty} \|B^*x_n \| =0.$$
As
   $$\| A\|=\lim_{n\to \infty} |\langle Ax_n,x_n\rangle |\leq \lim_{n\to \infty} \| Ax_n\| \leq \|A \| ,$$
one gets $\lim  	\limits_{n \to \infty} |\langle Ax_n,x_n\rangle |=\lim  	\limits_{n \to \infty}  \| Ax_n\| = \|A \|$. Write $\lim  	\limits_{n \to \infty} \langle Ax_n,x_n\rangle=\lambda $ with   $|\lambda |=w(A)$. Then
   $$\| Ax_n-\lambda x_n\|^2=\langle Ax_n-\lambda x_n,Ax_n-\lambda x_n\rangle= \|Ax_n \|^2-\bar\lambda \langle Ax_n,x_n\rangle -\lambda \langle x_n,Ax_n\rangle +|\lambda |^2.$$
It follows  that $\lim  	\limits_{n \to \infty} \|Ax_n-\lambda x_n \|=0 $.  Also note  that $\lim  	\limits_{n \to \infty} \|B^*Ax_n\|=\lim  	\limits_{n \to \infty} \|\lambda B^*x_n\|=0$. By \cite[Proposition 2.8]{AR}, we have $A\perp ^s_BB$.

Now, take $A=x\otimes y+y\otimes x$,  which is self-adjoint, and $B=x\otimes x$ with $x,y\in {\rm sphere}\,H$ and $\langle x,y\rangle =0$. Then  $\| Ax\|=\| A\|  $ and $B^*Ax=0$. This, by \cite[Proposition 2.8]{AR}, implies that $A\bot ^s_BB$.
However,   if $|\langle Au,u\rangle |=w(A)$ for some $u\in {\rm sphere}\,H$, then   $u=\frac{\sqrt{2}}{2}(x+y)$ or $\frac{\sqrt{2}}{2}(x-y)$. So  $B^*u\neq 0$. By Proposition \ref{prop2.1}, we see that $A \not\perp^s_w B$.
\end{proof}

\section{ The diameter of $\varGamma (\mathcal A, \perp\!\!\!\!\perp_w)$ with $\mathcal A=\mathcal{M}_2 (\mathbb C)$ or $\mathsf{H} _2$}

 It is known from  \cite{AG} that, for $\dim H\geq 2$,  both the numerical radius orthographs   $\varGamma (\mathcal{B} (H), $\\$\perp\!\!\!\!\perp_w)$ and $\varGamma (\mathcal{B}_s (H), \perp\!\!\!\!\perp_w)$  are connected, that is, they have no isolated vertices. It is  also known that  Diam\,$\varGamma (\mathcal{B} (H), \perp\!\!\!\!\perp_w)={\rm Diam}\, \varGamma (\mathcal{B}_s (H), \perp\!\!\!\!\perp_w)=2$ if $\dim H\geq 3$.  The purpose of this section is to give the diameters of these two graphs for the case when $\dim H=2$.

In fact, we have the following result.


\begin{thm}\label{prop6.3}
	{\rm Diam}$\ \varGamma (\mathcal{M}_2(\mathbb{C}), \perp\!\!\!\!\perp_w)=3$ and	{\rm Diam} $\varGamma (\mathsf{H} _2, \perp\!\!\!\!\perp_w)=2$.

\end{thm}

To prove this theorem, several propositions are needed.

\begin{prop}\label{prop6.1}
Assume that $x,y,u,v\in{\rm sphere}\,\mathbb{C} ^2$ with $\langle x,y\rangle =\langle u,v\rangle=0$. Then  $x\otimes y\perp_wu\otimes v$ if and only if $|\langle u,x\rangle |\leq |\langle u,y\rangle |$.
\end{prop}

\begin{proof} Take any $x,y,u,v\in{\rm sphere}\,\mathbb{C} ^2$ with $\langle x,y\rangle =\langle u,v\rangle=0$.
Without loss of generality, we can write  $u=\alpha x+\beta y$ with  $\alpha ,\beta \geq 0$; otherwise, consider $x\otimes e^{i(\delta -\theta) }y$ and $e^{-i\theta }u\otimes e^{i\delta }v$ instead, where $\theta =\arg \alpha $ and $\delta =\arg \beta $. Then  $v=\beta x-\alpha y$. With respect to the space decomposition $\mathbb{C} ^2=[x,y]$, one has
  $$x\otimes y=\left(
  \begin{array}{ccc}
      0 & 1     \\
      0   & 0
    \end{array}
  \right) {\rm \ and \ }u\otimes v=\left(
  \begin{array}{ccc}
      \alpha \beta  & -\alpha ^2     \\
      \beta ^2   & -\alpha \beta
    \end{array}
  \right).$$
  By Eq.(\ref{2.01}), we get
  $$\{z\in{\rm sphere}\,\mathbb{C} ^2:|\langle (x\otimes y)z,z\rangle |=w(x\otimes y)\}=\{z_\xi :=\frac{1}{\sqrt{2}}(x+e^{i\xi }y):\xi\in[0,2\pi)\}.$$ Furthermore, a simple calculation gives
  $$\overline{\langle (x\otimes y)z_\xi,z_\xi\rangle} \langle (u\otimes v)z_\xi,z_\xi\rangle=\frac{1}{4}(\beta ^2e^{-2i\xi }-\alpha ^2) .$$
  Now, by using  the fact
  $$\beta \geq \alpha \Leftrightarrow 0\in{\rm co}\,\{\beta ^2e^{-2i\xi }-\alpha ^2:\xi\in[0,2\pi)\},$$
one  obtains  that
  $$\begin{array}{rl}x\otimes y\perp_wu\otimes v&\Leftrightarrow 0\in{\rm co}\,\{\overline{\langle (x\otimes y)z_\xi,z_\xi\rangle} \langle (u\otimes v)z_\xi,z_\xi\rangle:\xi\in[0,2\pi)\}\\
    &\Leftrightarrow \beta \geq \alpha\Leftrightarrow |\langle u,y\rangle |\geq |\langle u,x\rangle |.\end{array}$$
The proof is finished.
\end{proof}

\begin{prop}\label{prop4.2}
Assume that $A=\left(
  \begin{array}{ccc}
      a_1 & a_2     \\
      a_3   & a_4
    \end{array}
  \right)\in\mathcal{M} _2(\mathbb{C} )$. If $|a_1|=w(A)$, then $|a_2|=|a_3|\leq |a_1|$.
\end{prop}

\begin{proof}
  Assume that $|a_1|=w(A)$. Then  $|a_2|=|a_3|$ by \cite[Lemma 1.5.13]{W2}.

  Still, we need to prove $|a_2|\leq |a_1|$.
Note that $W(A)$ is an ellipse with foci $\{\lambda _1,\lambda _2\}$ and a minor axis of the length $(|a_1|^2+|a_2|^2+|a_3|^2+|a_4|^2-|\lambda _1|^2-|\lambda _2|^2)^{\frac{1}{2}}$, where  $\sigma (A)=\{\lambda _1,\lambda _2\}$ (\cite{D}). Let the center of $W(A)$  be $\alpha +i\beta $ with $\alpha ,\beta \in\mathbb{R} $, the major axis length be $2a$ with direction $\theta \in[0,\pi)$, the minor axis length be $2b$, and the focal distance be $2c$.  A simple calculation gives
  $$\{\lambda_1,\lambda_2\}=\{\alpha +c\cos \theta +i(\beta +c\sin \theta ),\alpha -c\cos \theta +i(\beta -c\sin \theta )\}.$$
Using the length of the minor axis, we obtain
\begin{equation}\label{3.01}
  |a_1|^2+|a_2|^2+|a_3|^2+|a_4|^2=2a^2+2b^2+2\alpha ^2+2\beta ^2.
\end{equation}
As $a_1+a_4={\rm tr}\,A=\lambda _1+\lambda _2=2(\alpha +i\beta )$, we may set $a_1=\alpha +s+i(\beta +t)$ and $a_4=\alpha -s+i(\beta -t)$ with $s,t\in\mathbb{R} $. By the fact   $|a_2|=|a_3|$ and Eq.(\ref{3.01}), we infer that $|a_2|^2=a^2+b^2-s^2-t^2$.
This implies
\begin{equation}\label{3.02}
  |a_2|^2\leq a^2+b^2-b^2=a^2.
\end{equation}
In addition, note that  $\mu _1:=\alpha +a\cos \theta +i(\beta +a\sin \theta ),\mu _2:=\alpha -a\cos \theta +i(\beta -a\sin \theta )\in W(A)$. It follows that
$$\sqrt{\alpha ^2+\beta ^2+a^2+2a|\alpha \cos \theta +\beta \sin \theta |} =\max\{|\mu _1|,|\mu _2|\}\leq w(A),$$
and so   $a\leq w(A)$. Hence $|a_2|\leq w(A)=|a_1|$ by Eq.(\ref{3.02}).
\end{proof}

\begin{prop}\label{prop6.2}
 Assume that $A\in\mathcal{M} _2(\mathbb{C} )$. If $|\langle Ax,x\rangle |=w(A)$ for some $x\in{\rm sphere}\,\mathbb{C} ^2$, then
both  $A\perp\!\!\!\!\perp_wx\otimes y$  and $A\perp\!\!\!\!\perp_wy\otimes x$ hold for all $y\in{\rm sphere}\,\mathbb{C} ^2$  with $\langle x,y\rangle=0$.
\end{prop}

\begin{proof}
	
Take any 	$y\in{\rm sphere}\,\mathbb{C} ^2$  with $\langle x,y\rangle=0$ and $z_\delta =\frac{1}{\sqrt 2}(x+e^{i\delta }y)$.  We will complete the proof by considering  two cases.

  {\bf Case 1.} $A$ is normal.

  With respect to the space decomposition $\mathbb{C} ^2=[x,y]$,  $A$ can be written as $A=\left(
  \begin{array}{ccc}
      a_1 & 0     \\
      0   & a_2
    \end{array}
  \right)$ with  $|a_2|\leq |a_1|=w(A)$. It is easy to verify that $A\perp_wx\otimes y$.
In addition, we have
  $$\overline{\langle (x\otimes y)z_\delta ,z_\delta \rangle} \langle Az_\delta ,z_\delta \rangle=\frac{1}{4}e^{-i\delta } (a_1+a_2),$$
  which means $0\in{\rm co}\,\{\overline{\langle (x\otimes y)z_\delta ,z_\delta \rangle } \langle Az_\delta ,z_\delta \rangle:\delta \in[0,2\pi)\}$. It follows from Proposition \ref{1.1} and Eq.(\ref{2.01}) that $x\otimes y\perp_wA$, and thus $A\perp\!\!\!\!\perp_wx\otimes y$.

  Similarly, one can prove that $A\perp\!\!\!\!\perp_wy\otimes x$.

  {\bf Case 2.} $A$ is not normal.

  With respect to the space decomposition $\mathbb{C} ^2=[x,y]$,  $A$ can be written as $A=\left(
  \begin{array}{ccc}
      a_1 & a_2     \\
      a_3 & a_4
    \end{array}
  \right)$.  By Proposition \ref{prop4.2}, one has $|a_2|=|a_3|\leq |a_1|=|\langle Ax,x\rangle |=w(A)$. It is easy to verify that $A\perp_wx\otimes y$. In addition,  we have
  $$\overline{\langle (x\otimes y)z_\delta ,z_\delta \rangle} \langle Az_\delta ,z_\delta \rangle=\frac{1}{4}[a_3e^{-2i\delta }+(a_1+a_4)e^{-i\delta }+a_2].$$

  Let $f(u)=a_3u^2+(a_1+a_4)u+a_2$. To prove $x\otimes y\perp_wA$,  by Proposition \ref{1.1}, it suffices to show
  \begin{equation}\label{6.1}
  0\in{\rm co}\,\{f(e^{i\xi}):\xi\in[0,2\pi)\}.
  \end{equation}
 In fact, if there exists some $\xi\in[0,2\pi)$ such that $f(e^{i\xi})=0$, then Eq.(\ref{6.1}) holds trivially. If there does not exist $\xi\in[0,2\pi)$ such that $f(e^{i\xi})=0$, then let $\{u:f(u)=0\}=\{u_1,u_2\}$. Note that $|u_1u_2|=|\frac{a_2}{a_3}|=1$. Then there exists $i\in\{1,2\}$ such that $|u_i|<1$. Through a discussion similar to that of Eq.(\ref{2.011}), one achieves that Eq.(\ref{6.1}) holds. Hence $A\perp\!\!\!\!\perp_wx\otimes y$.

  Similarly, one can show $A\perp\!\!\!\!\perp_wy\otimes x$. The proof is finished.
\end{proof}

\begin{prop}\label{prop7.1}
If $A=\left(
	\begin{array}{ccc}
		1 & 0     \\
		0 & a
	\end{array}
	\right)\in \mathsf{H} _2$ with $|a|\leq 1$ and $B_\theta =\left(
	\begin{array}{ccc}
		0 & e^{i\theta }     \\
		e^{-i\theta } & 0
	\end{array}
	\right)$ with $\theta \in[0,2\pi)$, then  $A\perp\!\!\!\!\perp_wB_\theta$.
\end{prop}

\begin{proof}
On the one hand, by taking  $x=\left(
	\begin{array}{ccc}
		1    \\
		0
	\end{array}
	\right)$, one has
	$$|\langle Ax,x\rangle |=w(A) {\rm \ and \ }|\overline{\langle Ax,x\rangle}\langle B_\theta x,x\rangle |=0,$$
	which means $A\perp_wB_\theta $.
	
	On the other hand, by taking  $x_1=\frac{1}{\sqrt{2}}\left(
	\begin{array}{ccc}
		e^{i\theta }    \\
		1
	\end{array}
	\right)$ and $x_2=\frac{1}{\sqrt{2}}\left(
	\begin{array}{ccc}
		e^{i\theta }    \\
		-1
	\end{array}
	\right)$, a direct calculation gives
	$$|\langle B_\theta x_1,x_1\rangle |=|\langle B_\theta x_2,x_2\rangle |=w(B)$$
	and
	$$0\in{\rm co}\,\{\overline{\langle B_\theta x_1,x_1\rangle} \langle A x_1,x_1\rangle , \overline{\langle B_\theta x_2,x_2\rangle} \langle A x_2,x_2\rangle \}={\rm co}\,\{\frac{1+a}{2},-\frac{1+a}{2}\}.$$
	These imply  $B_\theta\perp_wA $. Hence $A\perp\!\!\!\!\perp_wB_\theta $.
\end{proof}

Now, we are at the position to give a proof of Theorem \ref{prop6.3}.

\begin{proof}[Proof of Theorem \ref{prop6.3}]
  Firstly, for any $A,B\in\mathcal{M}_2(\mathbb{C})$, there exist unit vectors $x,u\in H$ such that $|\langle Ax,x\rangle |=w(A)$ and $|\langle Bu,u\rangle |=w(B)$. It follows from Proposition \ref{prop6.2} that
  $$A\perp\!\!\!\!\perp_wx\otimes y, \ A\perp\!\!\!\!\perp_wy\otimes x, \ B\perp\!\!\!\!\perp_wu\otimes v {\rm \ and \ }B\perp\!\!\!\!\perp_wv\otimes u,$$
  where $y,v\in{\rm sphere}\,\mathbb{C} ^2$ with $\langle x,y\rangle =\langle u,v\rangle=0 $. By Proposition \ref{prop6.1},  if $|\langle x,u\rangle |\leq |\langle y,u\rangle |$, we have
  $$A\perp\!\!\!\!\perp_wx\otimes y\perp\!\!\!\!\perp_wu\otimes v\perp\!\!\!\!\perp_wB;$$
if $|\langle x,u\rangle |> |\langle y,u\rangle |$,    one has
  $$A\perp\!\!\!\!\perp_wy\otimes x\perp\!\!\!\!\perp_wu\otimes v\perp\!\!\!\!\perp_wB.$$
Hence ${\rm Diam}\,\varGamma (\mathcal{M}_2(\mathbb{C}), \perp\!\!\!\!\perp_w)\leq3$.

 So, to prove ${\rm Diam}\,\varGamma (\mathcal{M}_2(\mathbb{C}), \perp\!\!\!\!\perp_w)=3$, it suffices to check that   there exist $A,B\in\mathcal{M}_2(\mathbb{C}) $ such that the distance between them is greater than $2$.
To do this, let $A=\left(
  \begin{array}{ccc}
      5 & i     \\
      i & 1
    \end{array}
  \right)$ and $B=\left(
  \begin{array}{ccc}
      2 & 1     \\
      0 & 4
    \end{array}
  \right)$.
A simple calculation yields that
 $W(A)$ is an ellipse with the center $3$, the major axis length $4$, and the  minor axis length $2$. So $0\not\in W(A)$, $w(A)=5$ and $$\{x\in{\rm sphere}\,\mathbb{C} ^2:|\langle Ax,x\rangle |=w(A)\}=\{e^{i\delta }x_0:x_0 =\left(
  \begin{array}{ccc}
      1    \\
       0
    \end{array}
  \right),\delta \in[0,2\pi)\}.$$ It follows that
  $\overline{\langle Ax_0 ,x_0 \rangle}\langle Bx_0 ,x_0\rangle=10,$
  which means
\begin{equation}\label{claim1}
	A\not\perp_wB.
	\end{equation}

  {\bf Claim.} If $A\perp\!\!\!\!\perp_wC\perp\!\!\!\!\perp_wB$ for some $C\in\mathcal{M} _2(\mathbb{C} )$, then  $C=0$.

  Write $C=\left(
  \begin{array}{ccc}
      c_1 & c_2     \\
      c_3 & c_4
    \end{array}
  \right)$. As $A\perp_wC$, we must have $c_1=0$. Suppose on the contrary that $C\neq 0$.

  {\bf Case 1.} The center of $W(C)$ is not at the origin.

  In this case, there exists some $x_0\in {\rm sphere}\,\mathbb{C} ^2$  such that $|\langle Cx_0,x_0\rangle |=w(C)$ and
  $$\{x\in{\rm sphere}\,\mathbb{C} ^2:|\langle Cx,x\rangle |=w(C)\}=\{e^{i\delta }x_0:\delta \in[0,2\pi)\}.$$ As $C\perp_wA$, we have $\langle Ax_0,x_0\rangle =0$ by Proposition \ref{1.1}, a contradiction  to $0\notin W(A)$.

  {\bf Case 2.} The center of $W(C)$ is at the origin, but $W(C)$ is not a circle.

  In this case, we must have $c_4=0$. Thus, $C$ can be written as $C=\left(
  \begin{array}{ccc}
      0 & c_2     \\
      c_3e^{i\theta } & 0
    \end{array}
  \right)$ with $c_2,c_3>0$ and $\theta \in[0,2\pi)$, which yields $w(C)=\frac{c_2+c_3}{2}$. Note  that
  $$\{x\in{\rm sphere}\,\mathbb{C} ^2:|\langle Cx,x\rangle |=w(C)\}=\{e^{i\xi }x_1:\xi \in[0,2\pi)\}\cup \{e^{i\zeta  }x_2:\zeta \in[0,2\pi)\}$$
  with $x_1=\frac{1}{\sqrt{2}}\left(
  \begin{array}{ccc}
      1    \\
       e^{\frac{i\theta }{2}}
    \end{array}
  \right)$ and $x_2=\frac{1}{\sqrt{2}}\left(
  \begin{array}{ccc}
      1    \\
       -e^{\frac{i\theta }{2}}
    \end{array}
  \right)$.

  For $A$, one has
  $$\langle Ax_1,x_1\rangle \overline{\langle Cx_1,x_1\rangle}=\frac{1}{2} e^{-\frac{i\theta }{2}}(c_2+c_3)(3+i\cos \frac{\theta }{2})$$
  and
  $$\langle Ax_2,x_2\rangle \overline{\langle Cx_2,x_2\rangle}=-\frac{1}{2} e^{-\frac{i\theta }{2}}(c_2+c_3)(3-i\cos \frac{\theta }{2}).$$
As $C\perp_wA$, we infer that $0\in{\rm co}\,\{\langle Ax_1,x_1\rangle \overline{\langle Cx_1,x_1\rangle},\langle Ax_2,x_2\rangle \overline{\langle Cx_2,x_2\rangle}\}$, which yields $\theta =\pi$. Thus,
for $B$, a calculation gives
  $$\langle Bx_1,x_1\rangle \overline{\langle Cx_1,x_1\rangle}=\frac{1}{4} e^{-\frac{i\pi }{2}}(c_2+c_3)(6+i)$$
  and
  $$\langle Bx_2,x_2\rangle \overline{\langle Cx_2,x_2\rangle}=-\frac{1}{4} e^{-\frac{i\pi }{2}}(c_2+c_3)(6-i),$$
  which means $0\notin{\rm co}\,\{\langle Bx_1,x_1\rangle \overline{\langle Cx_1,x_1\rangle},\langle Bx_2,x_2\rangle \overline{\langle Cx_2,x_2\rangle}\}$,   contradicting to $C\perp_wB$.

  {\bf Case 3.} $W(C)$ is  a  disc centered at the origin.

  At this point, up to a nonzero scalar multiple, either $C=\left(
  \begin{array}{ccc}
      0 & 1     \\
      0 & 0
    \end{array}
  \right)$ or $C=\left(
  \begin{array}{ccc}
      0 & 0     \\
      1 & 0
    \end{array}
  \right)$. 
Note that the center of $W(B)$ is not at the origin. As $B\perp_wC$, there exists a unit vector $x\in{\rm sphere}\,\mathbb{C} ^2$ such that
  $$|\langle Bx,x\rangle |=w(B){\rm \ and \ }\langle Cx,x\rangle =0,$$
  which implies $x=\left(
  \begin{array}{ccc}
      1    \\
      0
    \end{array}
  \right)$ or $\left(
  \begin{array}{ccc}
      0    \\
      1
    \end{array}
  \right)$. This is impossible.

Therefor, the claim is true.

Now, by Claim,  it is clear that ${\rm Diam}\,\varGamma (\mathcal{M}_2(\mathbb{C}), \perp\!\!\!\!\perp_w)=3$.

Next, let us check that  ${\rm Diam}\,\varGamma (\mathsf{H}_2, \perp\!\!\!\!\perp_w)=2$.

 For any $A,B\in\mathsf{H} _2$, write
  $$A=a_1x_1\otimes x_1+a_2x_2\otimes x_2\ \ {\rm  and}\ \ B=b_1y_1\otimes y_1+b_2y_2\otimes y_2,$$ where $x_1,x_2,y_1,y_2\in{\rm sphere}\,\mathbb{C} ^2$, $\langle x_1,x_2\rangle=\langle y_1,y_2\rangle=0 $, $a_1,a_2,b_1,b_2\in\mathbb{R} $, $|a_1|=w(A)$ and $|b_1|=w(B)$. With respect to the space decomposition $\mathbb{C} ^2=[x_1,x_2]$, take $C_\theta =\left(
  \begin{array}{ccc}
      0 & e^{i\theta }     \\
      e^{-i\theta } & 0
    \end{array}
  \right)$ with any $\theta \in[0,2\pi)$. By Proposition \ref{prop7.1}, one gets $A\perp\!\!\!\!\perp_wC_\theta $.

Write  $y_1=\alpha x_1+\beta x_2$ with $\alpha ,\beta \in\mathbb{C} $. Choose an appropriate $\theta _0\in[0,2\pi)$ such that $\langle C_{\theta_0}y_1,y_1 \rangle =2{\rm Re}\,e^{i\theta _0}\bar\alpha \beta =0$. Then under the space decomposition $\mathbb{C} ^2=[y_1,y_2]$, we have $B=\left(
  \begin{array}{ccc}
      b_1 & 0     \\
      0 & b_2
    \end{array}
  \right)$ and $C_{\theta_0}=\left(
  \begin{array}{ccc}
      0 & \mu      \\
      \bar\mu  & 0
    \end{array}
  \right)$ with $\mu \in\mathbb{C} $. This and Proposition \ref{prop7.1} imply that  $B\perp\!\!\!\!\perp_wC_{\theta_0} $.

Thus, we have proved that
  $A\perp\!\!\!\!\perp_wC_{\theta_0}\perp\!\!\!\!\perp_wB,$
  that is, ${\rm Diam}\,\varGamma (\mathsf{H} _2, \perp\!\!\!\!\perp_w)\leq 2$.
However, for $S=\left(
  \begin{array}{ccc}
      1 & 0     \\
      0 & \frac{1}{2}
    \end{array}
  \right)$ and $T=\left(
  \begin{array}{ccc}
      1 & 0     \\
      0 & \frac{1}{3}
    \end{array}
  \right)$, clearly, $S\not\perp_wT$. Hence ${\rm Diam}\,\varGamma (\mathsf{H} _2, \perp\!\!\!\!\perp_w)=2$.

  This completes the proof.
\end{proof}

\section{the strong numerical radius orthograph $\varGamma (\mathcal{M} _n(\mathbb{C} ), \perp\!\!\!\!\perp^s_w)$}

In this section, we discuss the strong numerical radius orthograph $\varGamma (\mathcal{M} _n(\mathbb{C} ), \perp\!\!\!\!\perp^s_w)$.

 The following result says that the set NI$(\mathcal M_n(\mathbb C))$ of all singular matrices is connected in $\varGamma (\mathcal{M} _n(\mathbb{C} ), \perp\!\!\!\!\perp^s_w)$ with $n\geq 3$.

\begin{prop}\label{prop4.3}
Assume that  $n\geq 3$ and $A\in{\rm NI}(\mathcal{M} _n(\mathbb{C} ))$ is a nonzero operator. Then there exist mutually orthogonal vectors $x,y\in{\rm sphere}\,\mathbb{C} ^n$ such that $x\otimes y\perp\!\!\!\!\perp^s_wA$.
\end{prop}

\begin{proof}
Assume that  $0\not=A\in{\rm NI}(\mathcal{M} _n(\mathbb{C} ))$.  Then there exist vectors $y,v\in{\rm sphere}\,\mathbb{C} ^n$ such that $|\langle Ay,y\rangle |=w(A)$ and $A^*v=0$. Let $Ay=\alpha y+\beta u $ and $v=ay+bu +cw$, where $\{y,u ,w\}$ are mutually orthogonal unit vectors  in $\mathbb C^n$ and $\alpha ,\beta ,a,b,c\in\mathbb{C} $. It follows that
  \begin{equation}\label{4.1}
  0=\langle y,A^*v\rangle =\langle Ay,v\rangle =\langle \alpha y+\beta u, ay+bu+cw\rangle =\alpha \bar a+\beta \bar b.
  \end{equation}
  Under the space decomposition $\mathbb{C} ^n=[y,u ]\oplus[y,u ]^\bot $, we have
  $A=\left(
  \begin{array}{ccc}
      \alpha  & a_{12}  & A_{13}   \\
      \beta    & a_{22}   & A_{23} \\
      0   & a_{32}   & A_{33}
    \end{array}
  \right)$. Write $A'=\left(
  \begin{array}{ccc}
      \alpha  & a_{12}  \\
      \beta    & a_{22}
    \end{array}
  \right)$,  and then   $W(A')\subseteq W(A)$ by \cite[Problems 8.21]{W2}. So $$|\alpha |\leq w(A')\leq w(A)=|\alpha |.$$ It follows that  $|\alpha |=w(A')$. By Proposition \ref{prop4.2}, one gets $|\beta |\leq |\alpha |$, which and  Eq.(\ref{4.1}) imply $|a|\leq |b|$. Obviously,
 \begin{equation}\label{4.2}
  |a|\leq \frac{1}{\sqrt{2}}.
  \end{equation}

 Take $x=\frac{1}{\| bu +cw\| }(bu +cw)$. Then $\langle x,y\rangle =0$ and $v=ay+\| bu +cw\|x$. Note that $A^*v=aA^*y+\| bu+cw\|A^*x=0$. So,  $|a|\|A^*y \|=\| bu +cw\|\|A^*x\| $. Combining it with Eq.(\ref{4.2}), we obtain $\|A^*y \|\geq \|A^*x \| $. Now, by Proposition \ref{prop4.1}, one has $x\otimes y\perp^s_w A$; and  by  Proposition \ref{prop2.0}, one gets  $A\perp^s_wx\otimes y$.  Hence $x\otimes y\perp\!\!\!\!\perp^s_wA$.
\end{proof}

By the above proposition, we can characterize the connected components of the graph $\varGamma (\mathcal{M} _n(\mathbb{C} )), \perp\!\!\!\!\perp^s_w)$ and give a upper bound of Diam\,$\varGamma ({\rm NI}(\mathcal{M} _n(\mathbb{C} )), \perp\!\!\!\!\perp^s_w)$    with $n\geq 3$.

 Recall that the general linear group, denote by GL$(n,\mathbb C)$, is the group of all invertible matrices in $\mathcal M_n(\mathbb C)$. Thus we have GL$(n,\mathbb C)=\mathcal{M} _n(\mathbb{C} )\backslash {\rm NI}(\mathcal{M} _n(\mathbb{C} ))$.

\begin{thm}\label{prop4.5}
Assume that  $n\geq 3$. A connected component of  $\varGamma (\mathcal{M} _n(\mathbb{C} ), \perp\!\!\!\!\perp^s_w)$ is either an isolated vertice or  $\varGamma ({\rm NI}(\mathcal{M} _n(\mathbb{C} )),\perp\!\!\!\!\perp^s_w)$. Moreover, 
$$\{A\in\mathcal{M} _n(\mathbb{C} ): A {\rm \ is \ an \ isolated \ vertex }\}={\rm GL}(n,\mathbb C).$$
\end{thm}

\begin{proof}
 The connectivity of $\varGamma ({\rm NI}(\mathcal{M} _n(\mathbb{C} )),\perp\!\!\!\!\perp^s_w)$ is clear by  Theorem \ref{prop4.3} and {\color{red}  Proposition \ref{prop4.1}.}

For any invertible $A\in\mathcal{M} _n(\mathbb{C} )$, if $A\perp\!\!\!\!\perp^s_w B$ for some $B\in\mathcal{M} _n(\mathbb{C} )$, then  $B\perp^s_w A$. This and Proposition \ref{prop2.3} imply $B=0$. So $A$ is isolated. \end{proof}

\if false Finally, for  any $A\in{\rm NI}(\mathcal{M} _n(\mathbb{C} ))$,  by Proposition \ref{prop4.3},  there exists some $B\in{\rm NI}(\mathcal{M} _n(\mathbb{C} ))$ such that $A\perp\!\!\!\!\perp^s_w B$. So $A$ is not  an  isolated vertex.

As desired.\fi  

\begin{thm}\label{prop4.4}
${\rm Diam}\, \varGamma ({\rm NI}(\mathcal{M} _n(\mathbb{C} )), \perp\!\!\!\!\perp^s_w)\leq4$ for  $n\geq 3$.
\end{thm}

\begin{proof}
Take any $A_1,A_2\in{\rm NI}(\mathcal{M} _n(\mathbb{C} ))$. By Proposition \ref{prop4.3},  there exist $x,y,u,v\in {\rm sphere}\,\mathbb{C} ^n$ with $\langle x,y\rangle =\langle u,v\rangle =0$ such that $x\otimes y\perp\!\!\!\!\perp^s_wA_1$ and $u\otimes v\perp\!\!\!\!\perp^s_wA_2$. Take any unit vectors $e_1\in[x,u]^\perp $ and $e_2\in[e_1]^\perp $. By Proposition \ref{prop4.1},  one achieves
  $$A_1\perp\!\!\!\!\perp^s_w x\otimes y\perp\!\!\!\!\perp^s_w e_1\otimes e_2\perp\!\!\!\!\perp^s_w u\otimes v\perp\!\!\!\!\perp^s_wA_2,$$
which implies ${\rm Diam}\,\varGamma ({\rm NI}(\mathcal{M} _n(\mathbb{C} )), \perp\!\!\!\!\perp^s_w)\leq4$.
\end{proof}

 We do not know if exactly ${\rm Diam}\,\varGamma ({\rm NI}(\mathcal{M} _n(\mathbb{C} )), \perp\!\!\!\!\perp^s_w)=4$ for all $3\leq n<\infty$. Moreover, we do not know what will happen to the strong numerical  radius orthgraph $\varGamma (\mathcal{B} (H), \perp\!\!\!\!\perp^s_w)$ nor  what  the diameter of $\varGamma ({\rm NI}(\mathcal{B} (H)), \perp\!\!\!\!\perp^s_w)$ is if $\dim H=\infty$.
  
  In the rest of this section, we discuss  the   orthograph $\varGamma (\mathcal{M} _2(\mathbb{C} ), \perp\!\!\!\!\perp^s_w)$, which is very different from the case when $n\geq 3$.

Let $${\rm NI}_1=\{A\in{\rm NI}(\mathcal{M} _2(\mathbb{C} )):A^2\neq 0\}\ \ {\rm and}\ \ {\rm NI}_2=\{A\in{\rm NI}(\mathcal{M} _2(\mathbb{C} )):A^2= 0\}.$$ It is obvious that ${\rm NI}(\mathcal{M} _2(\mathbb{C} ))={\rm NI}_1\cup {\rm NI}_2$.

We first give some properties of the elements in  ${\rm NI}_1$ and ${\rm NI}_2$.

\begin{prop}\label{prop3.2}
For any $A\in{\rm NI}_i$,  there exists some $B\in{\rm NI}_j$ such that $A\perp\!\!\!\!\perp^s_w B$, where $i\not=j\in\{1,2\}$.
\end{prop}

\begin{proof}

Take any $A\in{\rm NI}_1$.  Then
$A$ can be written as $A=x\otimes (\alpha x+\beta y)$, where $x,y\in{\rm sphere}\,\mathbb{C} ^2$, $\langle x,y\rangle =0$, $\alpha ,\beta \in\mathbb{C} $ and $|\alpha |^2+|\beta |^2=1$ with $\alpha\not=0$. Without loss of generality, we may assume that $\alpha ,\beta \geq 0$; otherwise, replace $A$ by
$e^{i\theta }A=|\alpha |x\otimes x+|\beta |x\otimes e^{i(\delta -\theta )}y$
with $\theta =\arg \alpha $ and $\delta =\arg \beta $.
Let $B=u\otimes z\in{\rm NI}_2$ with $$u=\sqrt{\frac{1-\alpha }{2}}x-\sqrt{\frac{1+\alpha }{2}}y,\ \ \ z=\sqrt{\frac{1+\alpha }{2}}x+\sqrt{\frac{1-\alpha }{2}}y.$$ A direct calculation gives  $|\langle Az,z\rangle|=w(A)=\frac{1+\alpha }{2} $ and $B^*z=(z\otimes u)z=0$, which implies $A\perp^s_w B$ by Proposition \ref{prop2.0}.
 As $\langle x,u\rangle =\sqrt{\frac{1-\alpha }{2}}$ and $\langle x,z\rangle =\sqrt{\frac{1+\alpha }{2}}$, noting that  $\alpha \neq 0$, we arrive at $|\langle x,u\rangle|<|\langle x,z\rangle|$, which,  entails $B\perp^s_w A$.
Hence  $A\perp\!\!\!\!\perp^s_w B$ by Proposition \ref{prop2.01}.

Take any  $A\in{\rm NI}_2$. Then $A=u\otimes z$ for some $u,z\in{\rm sphere}\,H$ and $\langle u,z\rangle =0$. Let $B=x\otimes (\alpha x+\sqrt{1-\alpha ^2} y)\in{\rm NI}_1$, where
$$x=\sqrt{\frac{1-\alpha }{2}}u+\sqrt{\frac{1+\alpha }{2}}z,\ \ y=\sqrt{\frac{1-\alpha }{2}}z-\sqrt{\frac{1+\alpha }{2}}u,\ \ \alpha \in(0,1).$$ By the same argument as the above, one can show that $A\perp\!\!\!\!\perp^s_w B$.
\end{proof}

 Now, by Proposition \ref{prop3.2} and a  similar argument to that of Theorem \ref{prop4.5},   the following result is immediate.

\begin{thm}\label{prop3.3}
A connected component of  $\varGamma (\mathcal{M} _2(\mathbb{C} ), \perp\!\!\!\!\perp^s_w)$ is either an isolated vertice or the set  $\varGamma ({\rm NI} (\mathcal{M} _2(\mathbb{C} )) , \perp\!\!\!\!\perp^s_w)$. Moreover,
$$\{A\in\mathcal{M} _2(\mathbb{C} ):A {\rm \ is \ an \ isolated \ vertice }\}={\rm GL}(2,\mathbb{C} ).$$
\end{thm}

The remaining part of this section is devoted to deriving the diameter of $\varGamma ({\rm NI} (\mathcal{M} _2(\mathbb{C} )) , \perp\!\!\!\!\perp^s_w)$.

\begin{prop}\label{prop3.4}
For any  $A,B\in{\rm sphere}\,{\rm NI}_1 $, if $A\perp\!\!\!\!\perp^s_w B$, then $ |{\rm tr}\, A|=|{\rm tr}\, B|$.
\end{prop}

\begin{proof}
	
	Take any  $A,B\in{\rm sphere}\,{\rm NI}_1 $. We can write
$$A=x\otimes (\alpha x+\beta y)\ \ {\rm and}\ \ B=u\otimes v, $$ where $x,y,u,v\in{\rm sphere}\,\mathbb{C} ^2$, $\langle x,y\rangle =0$, $\alpha ,\beta \in\mathbb{C} \backslash \{0\}$ and $|\alpha |^2+|\beta |^2=1$. Let $\alpha =e^{-i\delta }|\alpha | $ and $\beta =e^{-i\theta  }|\beta |$ with $\delta ,\theta \in[0,2\pi)$. Through a simple calculation, one obtains
$$\{z\in{\rm sphere}\,\mathbb{C} ^2:|\langle Az,z\rangle |=w(A)\}=\{e^{i\xi }z_0:\xi\in [0,2\pi)\}$$ with
 \begin{equation}\label{4.01}
  z_0=\sqrt{\frac{1+|\alpha |}{2}}x+\sqrt{\frac{1-|\alpha |}{2}}e^{i(\delta -\theta )}y.
 \end{equation}
Since $A\perp^s_wB$, we have $B^*z_0=(v\otimes u)z_0=0$, and so  $\langle z_0,u\rangle =0$. Then
\begin{equation}\label{4.02}
  u=\sqrt{\frac{1-|\alpha |}{2}}e^{-i(\delta -\theta )}x-\sqrt{\frac{1+|\alpha |}{2}}y.
   \end{equation}
 Then $v=\lambda u+\gamma z_0$ with $\lambda ,\gamma \in\mathbb{C} $ and $|\lambda |^2+|\gamma |^2=1$, that is, $B=u\otimes (\lambda u+\gamma z_0)$.  Write $\lambda =|\lambda |e^{-i\delta '}$ and $\gamma =|\gamma |e^{-i\theta  '}$. Similarly,  we have
$$\{w\in{\rm sphere}\,\mathbb{C} ^2:|\langle Bw,w\rangle |=w(B)\}=\{e^{i\xi }w_0:\xi\in [0,2\pi)\}$$ with $$w_0=\sqrt{\frac{1+|\lambda |}{2}}u+\sqrt{\frac{1-|\lambda  |}{2}}e^{i(\delta' -\theta' )}z_0.$$ By Eqs.(\ref{4.01})-(\ref{4.02}), one gets
  $$\begin{array}{rl}w_0=&(e^{-i(\delta -\theta )}\sqrt{\frac{1+|\lambda |}{2}}\sqrt{\frac{1-|\alpha  |}{2}}+e^{i(\delta' -\theta' )}\sqrt{\frac{1-|\lambda |}{2}}\sqrt{\frac{1+|\alpha  |}{2}})x+\\
    &(-\sqrt{\frac{1+|\lambda |}{2}}\sqrt{\frac{1+|\alpha  |}{2}}+e^{i(\delta '-\theta '+\delta -\theta )}\sqrt{\frac{1-|\lambda |}{2}}\sqrt{\frac{1-|\alpha  |}{2}})y.\end{array}$$
As $B\perp^s_wA$, we infer that $A^*w_0=((\alpha x+\beta y)\otimes x)w_0=0$, which yields $\langle w_0,x\rangle =0$. So
$$e^{-i(\delta -\theta )}\sqrt{\frac{1+|\lambda |}{2}}\sqrt{\frac{1-|\alpha  |}{2}}+e^{i(\delta' -\theta' )}\sqrt{\frac{1-|\lambda |}{2}}\sqrt{\frac{1+|\alpha  |}{2}}=0,$$ 
which forces $|\lambda |=|\alpha |$. As $|{\rm tr}\, A|=|{\rm tr}\, (x\otimes (\alpha x+\beta y))|=|\alpha |$ and $|{\rm tr}\, B|=|{\rm tr}\, (u\otimes (\lambda u+\gamma z_0))|=|\lambda |$, it follows that $|{\rm tr}\, A|=|{\rm tr}\, B|$.
\end{proof}

\begin{prop}\label{prop3.5}
 If $A,B\in\mathcal{\rm NI}_2 $, then either $A\perp\!\!\!\!\perp^s_w B$ or $A\perp\!\!\!\!\perp^s_w C\perp\!\!\!\!\perp^s_wB$ for some $C\in{\rm NI}_2$.
\end{prop}

\begin{proof} For  any $A,B\in\mathcal{\rm NI}_2 $,  without loss of generality, we may assume $\|A\|=\|B\|=1$. \if false and $A\not\perp\!\!\!\!\perp^s_w B$.\fi
Then $A=x\otimes y$ and $B=u\otimes v$ for some $x,y,u,v\in{\rm sphere}\,\mathbb{C} ^2$ with $\langle x,y\rangle =\langle u,v\rangle =0$ and $u=\alpha x+\beta y$ with some  $\alpha, \beta\in\mathbb{C} $. Also we may assume that $\alpha ,\beta \geq 0$ and thus $v= \beta x- \alpha y$.
Indeed, letting $\alpha =|\alpha |e^{i\theta }$ and $\beta =|\beta |e^{i\delta }$, we have
$$e^{-i\theta }u=|\alpha |x+|\beta|e^{i(\delta -\theta )}y {\rm \ and \ } e^{i\delta }v=|\beta |x-|\alpha |e^{i(\delta -\theta )}y.$$
Thus we may replace $x\otimes e^{i(\delta -\theta )}y$ and $e^{-i\theta }u\otimes e^{i\delta }v $ with $A$ and $B$, respectively.

 If $\alpha =0$, then  $u=y$, $v=x$, and so  $B=y\otimes x$. It is obvious that $A\perp\!\!\!\!\perp^s_w B$.

If  $\alpha =1$, then  $u=x$, $v=-y$, and so  $B=-x\otimes y$. It follows that $A\perp\!\!\!\!\perp^s_w C\perp\!\!\!\!\perp^s_wB$ with $C=y\otimes x$.

So, we may assume that  $0<\alpha < 1$. Take  $C=z\otimes w\in{\rm NI}_2 $ with $z,w\in{\rm sphere}\,\mathbb{C} ^2$, where $z=\lambda x+\gamma y$, $w=\gamma x-\lambda y$, $\lambda >0$, $\gamma <0$ and $|\lambda |<|\gamma |$. As $\lim\limits_{\lambda  \to (\frac{1}{\sqrt{2}})^-}\frac{4\lambda \gamma }{\gamma ^2-\lambda ^2} =-\infty ,$
we can choose suitable $\lambda $ such that $\frac{4\lambda \gamma }{\gamma ^2-\lambda ^2}<\frac{\alpha ^2-\beta ^2}{\alpha \beta }$. So $|\langle u,z\rangle |<|\langle v,z\rangle |$ and $|\langle u,z\rangle |<|\langle u,w\rangle |$,  which, together with Proposition \ref{prop2.01}, yields  $B\perp\!\!\!\!\perp^s_w C$.
On the other hand, as  $|\lambda |<|\gamma |$, one gets  that $|\langle z,x\rangle |<|\langle z,y\rangle |$ and $|\langle x,z\rangle |<|\langle x,w\rangle |$, which entails  $A\perp\!\!\!\!\perp^s_w C$ by Proposition \ref{prop2.01}.  Hence $A\perp\!\!\!\!\perp^s_w C\perp\!\!\!\!\perp^s_wB$.

The proof is finished.
\end{proof}

\if false
\begin{prop}\label{prop4.07}
There exist two vertices $A,B\in\varGamma ({\rm NI} (\mathcal{M} _2(\mathbb{C} )) , \perp\!\!\!\!\perp^s_w )$ such that the distance between $A$ and $B$ is greater than $3$.
\end{prop}
\fi

\begin{thm}\label{th3.6}
 ${\rm Diam}\,\varGamma ({\rm NI} (\mathcal{M} _2(\mathbb{C} )), \perp\!\!\!\!\perp^s_w )=4$.
\end{thm}

\begin{proof}
Take $A=x\otimes x, \ B=u\otimes u\in{\rm NI} (\mathcal{M} _2(\mathbb{C} ))$ with $x,u\in{\rm sphere}\,\mathbb{C} ^2$  such that
$$u=\alpha x+\beta y\ \ {\rm with} \ y\in{\rm sphere}\,\mathbb{C} ^2,\ \ \langle x,y\rangle =0,\ \ \alpha ,\beta \in\mathbb{C} \ {\rm and}\ |\alpha |> |\beta |\neq 0.$$ 

It is obvious that $A\not\perp\!\!\!\!\perp^s_wB$. So the distance between $A$ and $B$ is greater than 1.

Assume that there exists some $C=z\otimes w\in{\rm sphere}\,({\rm NI} (\mathcal{M} _2(\mathbb{C} ))  ) $ with $z,w\in{\rm sphere}\,\mathbb{C} ^2$ such that $A\perp\!\!\!\!\perp^s_w C\perp\!\!\!\!\perp^s_wB$. By  $A\perp^s_wC$ and $B\perp^s_wC$, one obtains  $C^*x=C^*u=0$, which means $\langle x,z\rangle =\langle u,z\rangle =0$, a contradiction. So the distance between $A$ and $B$ is greater than 2.

Now, assume that  there exist some  $C_1,C_2\in{\rm sphere}\,({\rm NI} (\mathcal{M} _2(\mathbb{C} ))  ) $ such that $A\perp\!\!\!\!\perp^s_w C_1\perp\!\!\!\!\perp^s_w C_2\perp\!\!\!\!\perp^s_wB$.
If $C_1,C_2\in{\rm sphere}\,{\rm NI} _1$, then by Proposition \ref{prop3.4},  we have $|{\rm tr}\,C_1|=|{\rm tr}\,C_2|=|{\rm tr}(A)|=1$. Thus  $C_1=z_1\otimes z_1$ and $C_2=z_2\otimes z_2$ for some $z_1,z_2\in{\rm sphere}\,\mathbb{C}^2$. It follows that $\langle x,z_1\rangle =\langle z_1,z_2\rangle =\langle u,z_2\rangle =0$, which leads to $\langle u,x\rangle =0$, a contradiction.
If $C_1\in{\rm sphere}\,{\rm NI}_2$ and $C_2\in{\rm sphere}\,{\rm NI} _1$, we can write $C_1=w\otimes z$ with  $w,z\in{\rm sphere}\,\mathbb{C} ^2$ and $\langle w,z\rangle =0$. On the one hand, since $A\perp^s_wC_1$, one has  $(z\otimes w)x=0$, and thus $\langle x,w\rangle =0$, which means $w\in[y]$. Combined with $C_1\in{\rm sphere}\,{\rm NI}_2$, we get $C_1=e^{i\theta }y\otimes x$ for some $\theta \in[0,2\pi)$. On the other hand, as $B\perp\!\!\!\!\perp^s_wC_2$, one has $|{\rm tr}\,C_2|=|{\rm tr}\,B|=1$ by Proposition \ref{prop3.4} and $C_2\in{\rm sphere}\,{\rm NI} _1$. Thus  $C_2=e^{i\delta }v\otimes v$ with $v\in{\rm sphere}\,[u]^\bot $ and $\delta \in[0,2\pi)$.  It follows  from $C_2\perp^s_wC_1$ that $(x\otimes y)v=0$, which means $\langle v,y\rangle =0$, a contradiction with $\alpha ,\beta \neq 0$.
If $C_1,C_2\in{\rm sphere}\,{\rm NI} _2$, then we can write  $C_1=w_1\otimes z_1$ and $C_2=w_2\otimes z_2$, where $w_1,w_2,z_1,z_2\in{\rm sphere}\,\mathbb{C} ^2$ and $\langle w_1,z_1\rangle=\langle w_2,z_2\rangle =0$. As $A\perp^s_wC_1$ and $B\perp^s_wC_2$, one gets  $(z_1\otimes w_1)x=(z_2\otimes w_2)u=0$, which implies $\langle x,w_1\rangle =\langle u,w_2\rangle =0$. That is, $w_1\in[y]$ and $w_2\in[u]^\bot $. Hence   $C_1=e^{i\theta _1}y\otimes x$ and $C_2=e^{i\theta _2}v\otimes u$ with $\theta _1,\theta _2\in[0,2\pi)$ and $v\in{\rm sphere}\,[u]^\bot $. As $C_1\perp^s_wC_2$, we achieve $|\langle v,y\rangle |\leq |\langle v,x\rangle |$ by Proposition \ref{prop2.01}, contradicting to $|\alpha |>|\beta |$.

Combining the above discussion, we conclude that the distance between $A$ and $B$ is greater than $3$. So ${\rm Diam}\,\varGamma ({\rm NI} (\mathcal{M} _2(\mathbb{C} )), \perp\!\!\!\!\perp^s_w )>3$.

 Take any $A,B\in{\rm NI} (\mathcal{M} _2(\mathbb{C} ))$.

 If $A,B\in{\rm NI} _1$, by Proposition \ref{prop3.2}, there exist some $C_1,C_2\in{\rm NI} _2$ such that $A\perp\!\!\!\!\perp^s_w C_1$ and $B\perp\!\!\!\!\perp^s_w C_2$. By Proposition \ref{prop3.5},   either $C_1\perp\!\!\!\!\perp^s_w C_2$, or $C_1\perp\!\!\!\!\perp^s_w C_3\perp\!\!\!\!\perp^s_wC_2$ for some $C_3\in{\rm NI} _2$. Thus, we get that either
$A\perp\!\!\!\!\perp^s_w C_1\perp\!\!\!\!\perp^s_w C_2\perp\!\!\!\!\perp^s_w B,$
or
$A\perp\!\!\!\!\perp^s_w C_1\perp\!\!\!\!\perp^s_w C_3\perp\!\!\!\!\perp^s_w C_2\perp\!\!\!\!\perp^s_w B.$

If  $A\in{\rm NI} _2$ and $B\in{\rm NI} _1$, by Proposition \ref{prop3.2} and Proposition \ref{prop3.5}, one obtains that either
$A\perp\!\!\!\!\perp^s_w C_1\perp\!\!\!\!\perp^s_w B $ for  some $C_1\in{\rm NI} _2,$
or
$A\perp\!\!\!\!\perp^s_w C_2\perp\!\!\!\!\perp^s_w C_1\perp\!\!\!\!\perp^s_w B $  for   some $C_1,C_2\in{\rm NI} _2.$

If $A,B\in{\rm NI} _2$, by Proposition \ref{prop3.5}, we still obtain that either
$A\perp\!\!\!\!\perp^s_w  B ,$
or
$A\perp\!\!\!\!\perp^s_w C\perp\!\!\!\!\perp^s_w B $  for  some   $C\in{\rm NI} _2.$

 The above arguments show that $3<{\rm Diam}\,\varGamma ({\rm NI} (\mathcal{M} _2(\mathbb{C} )), \perp\!\!\!\!\perp^s_w )\leq 4$. Therefore, we must have ${\rm Diam}\,\varGamma ({\rm NI} (\mathcal{M} _2(\mathbb{C} )), \perp\!\!\!\!\perp^s_w )= 4$, as desired.
\end{proof}

\section{the strong numerical radius orthograph $\varGamma (\mathcal{B} _s(H), \perp\!\!\!\!\perp^s_w)$}

In this section,  we study the strong numerical radius orthograph $\varGamma (\mathcal{B} _s(H), \perp\!\!\!\!\perp^s_w)$ with $2\leq \dim H\leq \infty$.

Our first result is a useful proposition.

\begin{prop}\label{prop5.1}
  If $\dim H\geq 3$, then there exist two vertices $A,B\in\varGamma ({\rm NI}(\mathcal{B} _s(H)), \perp\!\!\!\!\perp^s_w)$ such that the distance between $A$ and $B$ is greater than $2$.
Particularly, if $\dim H=3$, then there exist two vertices $A,B\in\varGamma (\mathcal{B} _s(H), \perp\!\!\!\!\perp^s_w)$ such that the distance between $A$ and $B$ is greater than $3$.
\end{prop}

\begin{proof}
 Assume that $\dim H\geq 3$. Take
$e_1,e_2,x_1,x_2\in {\rm sphere}\,H$ with  $\langle e_1,e_2\rangle =0$, $x_1\in[e_1,e_2]^\bot $ and  $x_2=\lambda e_1+\mu e_2$, $\lambda ,\mu \in\mathbb{C} \backslash \{0\}$.
Let $$A=e_1\otimes e_1 +\frac{1}{2}P_{[e_1,e_2]^\bot } {\rm \ and \ }B=x_1\otimes x_1+ \frac{1}{2}P_{[x_1,x_2]^\bot }.$$  From Proposition \ref{prop2.1}, it is easily seen  that $B\not\perp^s_w A$, which means $A\not\perp\!\!\!\!\perp^s_wB$.

  Suppose that $A\perp\!\!\!\!\perp^s_wC\perp\!\!\!\!\perp^s_wB$ for some $C\in{\rm sphere}\,(\mathcal{B} _s(H))$.  We claim
 $Ce_1=0$. If, on the contrary, $Ce_1\neq 0$, take $\{u_n\}\subseteq {\rm sphere}\,H$ such that $w(A)=\lim\limits_{n \to \infty}|\langle Au_n,u_n\rangle |$. Write $u_n=\alpha _ne_1+\beta _nz_n$ with $\alpha _n,\beta _n\in\mathbb{C} $ and $z_n\in{\rm sphere}\,[e_1]^\bot $. Without loss of generality, we may assume $\alpha _n,\beta _n\geq 0$; otherwise, consider $e^{-i\xi _n}u_n$ and $e^{i(\zeta _n-\xi _n)}z_n$ instead, where $\xi _n=\arg \alpha _n$ and $\zeta _n=\arg\beta _n$.  So  $$1=w(A)=\lim\limits_{n \to \infty}|\alpha _n^2+\beta _n^2\langle Az_n,z_n\rangle | \leq\lim _{n\rightarrow \infty}(\alpha _n^2+\frac{1}{2}\beta _n^2)\leq \lim _{n\rightarrow \infty}(\alpha _n^2+\beta _n^2) =1,$$
 which implies
 $\lim\limits _{n\rightarrow \infty}(\alpha _n^2+\frac{1}{2}\beta _n^2)= \lim\limits _{n\rightarrow \infty}(\alpha _n^2+\beta _n^2)$.
Hence $\lim\limits_{n \to \infty}\beta _n=0$, and thus $\lim\limits_{n \to \infty}\alpha  _n=1$.
A direct calculation gives  $\lim\limits_{n \to \infty}\|e_1- u_n\|=0 $, and so  $\lim\limits_{n \to \infty}\| Cu_n\|=\|Ce_1\| \neq 0$, a contradiction to $A\perp^s_wC$.

Similarly,  by  the relations $C\perp^s_wA$ and $C\perp\!\!\!\!\perp^s_wB$, one can show that  $Ce_2=\eta _1 e_2$, $Cx_1=0$ and $Cx_2=\eta _2x_2$ with $\eta _1,\eta _2\in\{-1,1\}$. It follows that
  $$\lambda =\langle x_2,e_1\rangle =\langle \eta _2Cx_2,e_1\rangle=\langle \eta _2x_2,Ce_1\rangle=0,$$
  a contradiction. Therefore, the above $C$ does not exist and the distance between $A$ and $B$ is greater than $2$.

Now, assume that $\dim H=3$. Let $[e_1,e_2,e_3]$ and $[x_1,x_2,x_3]$ be any two  orthonormal bases of $H$. Write  $x_1=\lambda e_1+\mu e_2+\gamma e_3$ and $x_2=\lambda' e_1+\mu' e_2+\gamma' e_3$ with $\lambda ,\lambda '\mu ,\mu ',\gamma ,\gamma '\in\mathbb{C} \backslash \{0\}$. Take
$$A=e_1\otimes e_1 +\frac{1}{2}e_3\otimes e_3{\rm \ and \ }B=x_1\otimes x_1 +\frac{1}{2}x_3\otimes x_3.$$

  {\bf Claim 1.} $A\not\perp\!\!\!\!\perp^s_wB$.

  If $A\perp^s_wB$, then   $Be_1=0$ by Proposition \ref{prop2.1},  contradicting to $\mu '\neq 0$.

  {\bf Claim 2.} There does not exist $C\in{\rm NI} (\mathcal{B} _s(H))\backslash \{0\}$ such that $A\perp\!\!\!\!\perp^s_wC\perp\!\!\!\!\perp^s_wB$.

  If there exists such $C$,  then  by $A\perp\!\!\!\!\perp^s_wC$,   one gets $Ce_1=0$ and $Ce_2=\eta e_2$ with $\eta \in\{-1,1\}$;  by $B\perp\!\!\!\!\perp^s_wC$, one has $Cx_1=0$ and $Cx_2=\eta 'x_2$ with $\eta '\in\{-1,1\}$. It follows that
  $$\lambda '=\langle x_2,e_1\rangle=\langle \eta 'Cx_2,e_1\rangle=\eta '\langle x_2,Ce_1\rangle =0,$$
  a contradiction.

  {\bf Claim 3.} There does not exist $C_1,C_2\in{\rm NI} (\mathcal{B} _s(H))\backslash \{0\}$ such that $A\perp\!\!\!\!\perp^s_wC_1\perp\!\!\!\!\perp^s_wC_2\perp\!\!\!\!\perp^s_wB$.

Otherwise, an argument similar to that in Claim 2 gives
  $$C_1e_1=C_2x_1=0, \ C_1e_2=\eta e_2 {\rm \ and \ } C_2x_2=\eta 'x_2 {\rm \ with \ }\eta ,\eta '\in\{-1,1\}.$$
So we can set $C_1=\eta e_2\otimes e_2+\xi e_3\otimes e_3 $ and $C_2=\eta' x_2\otimes x_2+\xi' x_3\otimes x_3 $ with $\xi ,\xi' \in\mathbb{C} $ and $|\xi |,|\xi' |\leq 1$.
If  $\xi'\neq 0$, we have  $x_1\in[e_2,e_3]$ by $C_1\perp^s_wC_2$, a contradiction to $\lambda \neq 0$; if $\xi\neq 0$,  we get $e_1\in[x_2,x_3]$ by $C_2\perp^s_wC_1$,  contradicting to $\langle x_1,e_1\rangle =\lambda \neq 0$; if $\xi= \xi'= 0 $, one has  $\langle e_2,x_2\rangle =0$ by $C_1\perp^s_wC_2$, a contradiction to $\mu ' \neq 0$.

Now, by Claims 1-3,   we see that the distance between $A$ and $B$ must be greater than $3$.
\end{proof}

 Applying  Proposition 5.1, we can draw out the diameter of $\varGamma ({\rm NI} (\mathcal{B} _s(H)), \perp\!\!\!\!\perp^s_w)$.

\begin{thm}\label{prop5.2}
If $\dim H\geq 4$, then ${\rm Diam}\,\varGamma ({\rm NI} (\mathcal{B} _s(H)), \perp\!\!\!\!\perp^s_w)=3$;  if $\dim H=3$, then  ${\rm Diam}\,\varGamma ({\rm NI} (\mathcal{B} _s(H)), \perp\!\!\!\!\perp^s_w)=4$.
\end{thm}

\begin{proof}
	
  Take arbitrary $A,B\in{\rm NI} (\mathcal{B} _s(H))$. Without loss of generality, assume that $w(A)=w(B)=1$, $1\in\sigma(A)$ and $1\in\sigma(B)$.
We will prove the theorem by considering three cases.

{\bf Case 1.} $\dim H=\infty$.

 Let $A=\int z{\rm d}E(z)$ and $B=\int z{\rm d}F(z)$ be the spectral decomposition of $A$ and $B$, respectively.  For any  $\varepsilon _i>0$, $i \in\{1,2,3,4\}$, let
$$H_1=\begin{cases}E([1-\varepsilon _1,1])\ & {\rm if}\ \ 1\notin\sigma _p(A),\\
  [e_1]\ &{\rm if}\ \ 1\in\sigma _p(A),
  \end{cases} \ \
  H_2=\begin{cases}E([-\varepsilon _2,\varepsilon _2])\ & {\rm if}\ \ 0\notin\sigma _p(A),\\
  [e_2]\ &{\rm if}\ \ 0\in\sigma _p(A),
  \end{cases}
  $$
$$H_3=\begin{cases}F([1-\varepsilon _3,1])\ & {\rm if}\ \ 1\notin\sigma _p(B),\\
  [e_3]\ &{\rm if}\ \ 1\in\sigma _p(B),
  \end{cases} \ \
  H_4=\begin{cases}F([-\varepsilon _4,\varepsilon _4])\ & {\rm if}\ \ 0\notin\sigma _p(B),\\
  [e_4]\ &{\rm if}\ \ 0\in\sigma _p(B),
  \end{cases}
  $$
  where  $e_1\in {\rm sphere}\,E(1)$, $e_2\in{\rm sphere}\, E(0)$, $e_3\in{\rm sphere}\, F(1)$, $e_4\in{\rm sphere}\, F(0)$, $\sigma _p(X)$ denotes the point spectrum of $X$. Let $\varepsilon _i$ be sufficiently small so that there exist unit vectors $y\in[H_1,H_4]^\perp $ and $z\in[H_3,H_2,y]^\perp $. Let $C_1=P_{[H_2,y]}$ and $C_2=P_{[H_4,z]}$. Then we have
  $$A\perp\!\!\!\!\perp^s_wC_1\perp\!\!\!\!\perp^s_wC_2\perp\!\!\!\!\perp^s_wB,$$
  which, together with  Proposition \ref{prop5.1}, implies   ${\rm Diam}\,\varGamma ({\rm NI} (\mathcal{B} _s(H)), \perp\!\!\!\!\perp^s_w)=3$.

  {\bf Case 2.}   $4\leq\dim H<\infty$.

In this case, the numerical range $W(X)$ is a closed convex set for any $X\in\mathcal{B} (H)$. Thus, there exist $e_1,e_2,x_1,x_2\in{\rm sphere}\,H$ such that
  $$|\langle Ae_1,e_1\rangle |=w(A), \ Ae_2=0, \ |\langle Bx_1,x_1\rangle| =w(B) \ {\rm and  } \ Bx_2=0. $$
  Take $y\in[e_1,x_2]^\bot $,  $z\in[x_1,e_2,y]^\bot $, and let $C_1=P_{[e_2,y]}$ and $C_2=P_{[x_2,z]}$. It follows that
  $A\perp\!\!\!\!\perp^s_wC_1\perp\!\!\!\!\perp^s_wC_2\perp\!\!\!\!\perp^s_wB,$
  which again implies  ${\rm Diam}\,\varGamma ({\rm NI} (\mathcal{B} _s(H)), \perp\!\!\!\!\perp^s_w)=3$ by Proposition \ref{prop5.1}.

{\bf Case 3.} $\dim H=3 $.

Take $x,y\in{\rm sphere}\,H$ such that $Ax=By=0$. Let $C_1=x\otimes x$, $C_2=y\otimes y$ and  $C_3=z\otimes z$ with $z\in{\rm sphere}\,[x,y]^\bot $. It is easily checked that
  $A\perp\!\!\!\!\perp^s_wC_1\perp\!\!\!\!\perp^s_wC_3\perp\!\!\!\!\perp^s_wC_2\perp\!\!\!\!\perp^s_wB,$
which entails that  ${\rm Diam}\,\varGamma ({\rm NI} (\mathcal{B} _s(H)), \perp\!\!\!\!\perp^s_w)=4$ by
Proposition \ref{prop5.1}.
 
 The proof is complete. \end{proof}

  For the connected components of   $\varGamma (\mathcal{B} _s(H), \perp\!\!\!\!\perp^s_w)$, we have the following result.

 \begin{thm}\label{prop5.5} For $\dim H\geq 3$, the connected components of   $\varGamma (\mathcal{B} _s(H), \perp\!\!\!\!\perp^s_w)$ are either isolated vertices or  $\varGamma ({\rm NI} (\mathcal{B} _s(H)), \perp\!\!\!\!\perp^s_w)$.
Moreover, $$\{A\in\mathcal{B} _s(H): A {\rm \ is \ an \ isolated \ vertex }\}=\mathcal {B} _s(H)\backslash {\rm NI} (\mathcal{B} _s(H)).$$


\end{thm}

\begin{proof}
	Take any $A\in \mathcal{B} _s(H)$.

    If $A $ is invertible and $A\perp\!\!\!\!\perp^s_wB$ for some $B\in\mathcal{B} _s(H)$,  we have $B\perp^s_wA$, which and Proposition \ref{prop2.4} imply $B=0$. So  $A$ is an isolated vertex.

  If $A$ is not invertible, then  $0\in\sigma (A)$ and $0\in\overline{W(A)}=[\lambda _1,\lambda _2] $, where $\lambda_1=\min\{\lambda: \lambda\in\sigma(A)\}$ and $\lambda_2=\max\{\lambda: \lambda\in\sigma(A)\}$. Let $A=\int z{\rm d}E(z)$ be the spectral decomposition of $A$ . We may assume $|\lambda _2|\geq |\lambda _1|$; otherwise, consider $-A$ instead. Then there exist $\varepsilon,\varepsilon ' \geq 0$ such that the nontrivial subspaces $H_1=E([-\varepsilon ,\varepsilon ])$ and $H_2=E([\lambda _2-\varepsilon' ,\lambda _2])$ are orthogonal to each other. Take $B=P_{[H_1]}$. Then $A\perp\!\!\!\!\perp^s_wB$ and $B\neq 0$. This means that $A$ is not an isolated vertex.

Now, by the proof of Theorem \ref{prop5.2}, $\varGamma ({\rm NI} (\mathcal{B} _s(H)), \perp\!\!\!\!\perp^s_w)$ is connected.\end{proof}

 The case where the dimension of $H$ is 2 is different and needs to be discussed separately. In this case, $\mathcal B_s(H)=\mathsf{H} _2$.

 \begin{thm}\label{prop5.6} 
 Every   connected component  of   $\varGamma (\mathsf{H} _2, \perp\!\!\!\!\perp^s_w)$ is either an isolated vertice determined by an invertible matrix in $\mathsf{H} _2$,  or a set  of the form
$$S_x=\{u\otimes u:u\in[x]\cup [x]^\bot \}$$
for some $x\in {\rm sphere}\, \mathbb C^2$. Moreover, for each $x\in {\rm sphere}\, \mathbb C^2$, we have ${\rm Diam}\,\varGamma (S_x,\perp\!\!\!\!\perp^s_w)=1$.
\end{thm}

\begin{proof} Obviously, every invertible matrix in $\mathsf{H} _2$ is an isolated vertice of the orthgraph    $\varGamma (\mathsf{H} _2, $\\$\perp\!\!\!\!\perp^s_w)$. Assume $A,B\in \mathsf{H} _2$ are not invertible. Without loss of generality, we may assume that both $A$ and $B$ are positive. Then  $A=x\otimes x$ and $B=y\otimes y$ for some $x,y\in \mathbb C^2$. It is easy to verify that $A\perp\!\!\!\!\perp^s_w B\Leftrightarrow \langle x,y\rangle =0$. So $A$ is not an isolated vertice of $\varGamma (\mathsf{H} _2, \perp\!\!\!\!\perp^s_w)$, and  the connected components containing $A=x\otimes x$ has the form $S_x=\{u\otimes u: u\in[x]\cup [x]^\bot \} {\rm \ with \ }x\in H$.
Moreover,
for any $A,B\in S_x$, we have that either $A\perp\!\!\!\!\perp^s_w B$ or $A=\lambda B$ for some $\lambda \in\mathbb{R} $. Hence Diam\,$\varGamma (S_x,\perp\!\!\!\!\perp^s_w)=1$.

The proof is completed.
\end{proof}


\end{document}